\documentclass[12pt, reqno]{amsart}
\usepackage{amssymb, amsthm, amsmath, amsfonts}
\usepackage{array, epsfig}
\usepackage[unicode]{hyperref}
\usepackage[numbers,square]{natbib}
\usepackage{color}

  \evensidemargin
\oddsidemargin
\newcommand{\dom}{\mathop{\mathrm{dom}}\nolimits}

\renewcommand{\phi}{\varphi}

\newcommand{\bF}{\mathbb{F}}

\newcommand{\bL}{\mathbb{L}}

\newcommand{\bQ}{\mathbb{Q}}

\newcommand{\bS}{\mathbb{S}}

\newcommand{\bH}{\mathbb{H}}

\newcommand{\cC}{\mathcal{C}}

\newcommand{\cG}{\mathcal{G}}
\newcommand{\cH}{\mathcal{H}}

\newcommand{\cO}{\mathcal{O}}
\newcommand{\cU}{\mathcal{U}}

\newcommand{\cT}{\mathcal{T}}
\newcommand{\cX}{\mathcal{X}}

\newcommand{\N}{\mathbb{N}}

\newcommand{\Z}{\mathbb{Z}}

\theoremstyle{plain}
\newtheorem{theorem}{Theorem}[section]
\newtheorem{lemma}[theorem]{Lemma}
\newtheorem{corollary}[theorem]{Corollary}
\newtheorem{proposition}[theorem]{Proposition}

\theoremstyle{definition}
\newtheorem{definition}[theorem]{Definition}

\theoremstyle{remark}
\newtheorem{remark}[theorem]{Remark}

\begin{document}

\author{Zakhar Kabluchko}
\address{Zakhar Kabluchko: Institut f\"ur Mathematische Stochastik,
Westf\"alische Wilhelms-Universit\"at M\"unster,
Orl\'eans-Ring 10,
48149 M\"unster, Germany}
\email{zakhar.kabluchko@uni-muenster.de}

\author{Katrin Tent}
\address{Katrin Tent:
Mathematisches Institut und
Institut f\"ur Mathematische Logik und Grundlagenforschung,
Westf\"alische Wilhelms-Universit\"at M\"unster,
Einsteinstr.\ 62,
48149 M\"unster,
Germany}
\email{tent@wwu.de}

\title[]{Universal-homogeneous structures are generic}




\begin{abstract}
We prove that the Fra\"iss\'e limit of a  Fra\"iss\'e class $\cC$  is the (unique) countable structure whose isomorphism type is comeager (with respect to a certain logic topology) in the Baire space of all structures whose age is contained in $\cC$ and which are defined on a fixed countable universe.

In particular, the set of groups isomorphic to  Hall's universal group is comeager in the space of all countable locally finite groups and the set of fields isomorphic to the algebraic closure of $\bF_p$ is comeager in the space of  countable fields of characteristic $p$.
\end{abstract}

\maketitle

\section{Introduction}
In many cases,  universal homogeneous structures turn out to be ``typical'' in the space of all structures.  This can be understood either in the probabilistic sense of measure, or in the topological sense of category. In the former case, one constructs a probability measure on the space of all structures and proves that the collection of all structures isomorphic to the universal homogeneous one has measure $1$. Well-known examples are given by the random graph~\cite{erdoes_renyi,rado} and the Urysohn universal metric space which was constructed as a random object by Vershik~\cite{vershik_random_and_universal,vershik04,vershik06}. For an $\omega$-categorical countable $L$-structure $M$, it was proved  by~\citet{Droste_Kuske} that there is a random procedure that will almost surely construct $M$; see also~\cite{droste_goebel} and~\cite{droste_causal} for further examples.  Furthermore, by \cite{Ackerman_etal},  definable closure on $M$ is trivial if and only if
there is an $S_\infty$-invariant (or: exchangeable) probability measure on the space of $L$-structures with the set of structures isomorphic to $M$ having measure $1$.

Passing to the typicalness in the sense of category, one may ask whether the universal homogeneous structures form a comeager set in the space of all structures. Recall that a subset of a topological space is \emph{meager} (or a set of the first category) if it can be represented as a countable union of nowhere dense sets.  A set is called \emph{comeager} (or a set of the second category) if its complement is meager.   For the isomorphism type of the universal homogeneous graph and the rational Urysohn metric space, the comeagerness was proved by Vershik~\cite{vershik_random_and_universal,vershik04,vershik06}.

The aim of the present note is to prove the comeagerness of the isomorphism type of the Fra\"iss\'e limit (or the universal homogeneous structure) of an arbitrary Fra\"iss\'e class.  As a corollary, this gives an alternative proof of the existence of the Fra\"iss\'e limit.
While the statement that the Fra\"iss\'e limit is ``generic'' probably  belongs to the mathematical folklore and is well known to specialists, the  written proofs we are aware of impose some additional assumptions by considering only finite or relational structures; see \citet[(4.28)]{cameron_book}, \citet[Theorem~7.1]{cameron_age}, \citet{bankston_ruitenburg}, \citet{pouzet}, \citet{ivanov}, \citet[Theorem~2]{goldstern}, \citet[Theorems 4.2.2 and 4.1.13]{kruckman}, \citet{blog}, or use the setting of Banach--Mazur games on posets~\citet{kubis}, \citet{krawczyk_kubis}.   In Section~\ref{sec:main} we shall state and prove the main result under the same assumptions as the classical Fra\"iss\'e theorem, so that the case of finitely generated (rather than just finite) structures is covered. Compared to the case of finite structures, dealing with finitely generated structures requires introducing another topology on the space of structures, which is probably the main point of the present note.   Examples will be discussed in Section~\ref{sec:examples}.



\section{Background and the main result}\label{sec:main}
\subsection{Fra\"ss\'e limits}

Let $L$ be a countable language. For background on model theory we refer to \cite{TZ}.

\begin{definition}\label{def:Fraisse} Let $\cC$ be an infinite, countable class of (isomorphism types of) finitely generated $L$-structures.
We say that  $\cC$ is  a  \emph{Fra\"iss\'e class} if $\cC$ satisfies
\begin{itemize}
\item[(AP)] (Amalgamation Property) For any $A, B_1, B_2\in\cC$ and embeddings $h_i$ of $A$ into $B_i, i=1,2$, there is some $C\in\cC$ and embeddings $g_i$ of $B_i$ into $C, i=1,2$, such that $g_1\circ h_1=g_2\circ h_2$.
\item[(JEP)] (Joint Embedding Property) For any $A, B\in\cC$ there is some $D\in\cC$ containing $A$ and $B$ as a substructure.
\item[(HP)]  (Hereditary Property) For $A\in\cC$, any finitely generated substructure of $A$ is in~$\cC$.
\end{itemize}
\end{definition}

For an $L$-structure $M$, the \emph{age} of $M$ is the class of isomorphism types of finitely generated $L$-substructures of $M$.

Fra\"iss\'e's theorem states that for any Fra\"iss\'e class $\cC$ there is a countable structure $M$, the \textit{Fr\"aiss\'e limit} of $\cC$, with the following two  properties:

\begin{itemize}
\item (Universality) The age of $M$ is $\cC$.
\item (Homogeneity) Any isomorphism between finitely generated substructures of $M$ extends to an automorphism of $M$.
\end{itemize}

Moreover, $M$ is unique up to isomorphism.

\begin{remark}\label{rem:saturated}[\cite{TZ} Thm. 4.4.2]
A countable $L$-structure $M$ is the Fr\"aiss\'e limit of a countable class $\cC$ of finitely generated $L$-structures if and only if the following holds:
\begin{itemize}
\item[1.] The age of $M$ is $\cC$.
\item[2.] The structure $M$ is $\cC$-saturated, i.e.\ for any structures $A,B\in\cC$ with $A\subset B$ and any embedding $h$ of $A$ into $M$ there is an embedding of $B$ into $M$ extending $h$.
\end{itemize}
Furthermore, it is easy to see by the usual back-and-forth argument that these two conditions determine a countable $L$-structure uniquely up to isomorphism.

\end{remark}

Note that if $\cC$ consists of finite structures, then $M$ is locally finite, i.e.\ all finitely generated substructures of $M$ are finite.

Here and in what follows, the word `structure' always means `$L$-structure'. The set on which a structure $A$ is defined is denoted by $\dom A$ and called the universe or the domain of the structure.

\subsection{The space of structures}
Let $\cC$ be a Fra\"iss\'e class of structures whose Fra\"iss\'e limit does not belong to $\cC$.\footnote{M\'ark Po\'or gave an example of a  Fra\"iss\'e class $\cC$ whose  Fra\"iss\'e limit is finitely generated and thus belongs to $\cC$. We thank Tam\'as K\'atay for pointing out that we need to assume for our results that the Fra\"iss\'e limit is not finitely generated.}

As a common universe for our structures, we fix some countably infinite set, for example $\omega$.
Let $\bS$ be the set of all structures $S$ defined on the universe $\omega$ with the following two properties:
\begin{itemize}
\item[(i)] the age of $S$ is contained in $\cC$ and
\item[(ii)] $S$ is not finitely generated.
\end{itemize}
 We introduce the following topology on $\bS$.  Take some finitely generated structure $B$ whose isomorphism type is contained in $\cC$ and such that $\dom B\subset \omega$.
Let $\cO_B$ be the set of all structures $A\in \bS$ whose restriction to $\dom B$ coincides with $B$.
\begin{lemma}\label{lem:O}
$\cO_B$ is non-empty if and only if $\omega \backslash \dom B$ is infinite.
\end{lemma}
\begin{proof}
If $\omega \backslash \dom B$ is finite, then any structure $S$ with $\dom S=\omega$ which coincides with $B$ on $\dom B$ is finitely generated by the generators of $B$ together   with $\omega \backslash \dom B$. Hence, $S\notin \bS$ and $\cO_B$ is empty.

Suppose that $\omega \backslash \dom B$ is infinite. First observe that every structure $C\in \cC$ can be embedded into a strictly larger structure $C_1\in \cC$. Indeed, if the age of $C$ is strictly smaller than $\cC$, we can take some structure $A\in \cC$ not contained in the age of $C$, and use the joint embedding property to construct a new structure $C_1$ containing copies of $A$ and $C$. Clearly, $C_1$ is strictly larger than $C$ because $A$ can be embedded into $C_1$ but not into $C$.  On the other hand, if the age of $C$ equals $\cC$ then $C$ cannot be homogeneous because $C$ is not the Fra\"iss\'e limit. Therefore we can find isomorphic finitely generated substructures $A$ and $B$ of $C$ and an isomorphism $f:A\longrightarrow B$ which does not extend to an automorphism of $C$. Thus, there is an element $a\in C$ such that there is no $b\in C$ such that $f$ extends to the substructure  $D$ generated by $A$ and $a$. Since $A$ and $B$ are isomorphic, we can amalgamate $C$ with $D$ over $B$ to obtain a structure strictly larger than $C$.

Using the above, we can construct a sequence of embedded structures $B = B_1 \subsetneqq B_2 \subsetneqq \ldots$ such that $B_n\in\cC$ for all $n\in\N$.  If $\omega \backslash \dom B$ is infinite, we can embed these structures into $\omega$ such that  $\dom B = \dom B_1 \subsetneqq \dom B_2 \subsetneqq \ldots\subset \omega$ and the union of these domains is $\omega$. Now let $S$ be the inductive limit of this sequence. The age of $S$ is contained in $\cC$ because any finitely generated substructure of $S$ is contained in some $B_n$. For the same reason and since $\dom B_n \neq \omega$, the structure $S$ is not finitely generated. Hence, $S\in \bS$ and $S\in  \cO_B$.
\end{proof}

Those sets of the form $\cO_B$ which are non-empty form a base of some topology on $\bS$ which is denoted by $\cT$. Indeed, if $A \in \cO_{B_1}\cap \cO_{B_2}$ for some sets $\cO_{B_1}$ and $\cO_{B_2}$, then $A\in \cO_B \subset \cO_{B_1}\cap \cO_{B_2}$, where $B$ is the structure generated by $B_1\cup B_2$ inside $A$. Note that the isomorphism type of $B$ is contained in $\cC$ because $B$ is a finitely generated substructure of $A\in \bS$.  Observe also that the sets $\cO_B$ cover $\bS$ because every structure $A\in \bS$ has some finitely generated substructure.

\begin{proposition}\label{prop:baire}
The topological space $(\bS, \cT)$ has the Baire property, namely if $U_1,U_2,\ldots$ are open dense subsets of $\bS$, then their intersection is dense in $\bS$.
\end{proposition}
\begin{proof}
Let $B_0$ be  some structure whose isomorphism type is contained in $\cC$ and such that $\dom B_0 \subset \omega$, $|\omega \backslash \dom B_0| = \infty$. We have to exhibit a structure $A\in  \bS$ that is contained in $U_1,U_2,\ldots$ and in $\cO_{B_0}$. Since $U_1$ is dense, there is some  $A_0\in \cO_{B_0}\cap U_1\subset \bS$. Moreover, since the set $\cO_{B_0}\cap U_1$ is open, there is some basic open set $\cO_{B_1}$ such that $A_0\in \cO_{B_1} \subset \cO_{B_0}\cap U_1$. Without loss of generality we may assume that $\dom B_0\subset \dom B_1$ and $1\in \dom B_1$ since otherwise we could replace $B_1$ by the substructure of $A_0$ generated by $B_0\cup B_1\cup \{1\}$. Note also that $|\omega \backslash B_1| = \infty$ because $A_0\in\bS$ is not finitely generated.  Proceeding in this way, we can construct a sequence of embedded structures $B_0, B_1, \ldots$ whose isomorphism types are contained in $\cC$, such that $\dom B_0 \subset \dom B_1 \subset \ldots$  and   $\cO_{B_n} \subset \cO_{B_{n-1}} \cap U_n$, $n\in \dom B_n$, $|\omega\backslash \dom B_n| = \infty$ for all $n=1,2,\ldots$. Let $A$ be a structure obtained as the inductive limit of $B_0,B_1,\ldots$. The universe of $A$ is $\omega$ since $n\in \dom B_n$ for all $n$. Then, $A\in \bS$ because any finitely generated substructure of $A$ is a substructure of some $B_n$ and hence the age of $A$ is contained in $\cC$. Also,  $A$ itself is not finitely generated because $A\neq B_n$, which follows from $|\omega\backslash \dom B_n| = \infty$.  Clearly, $A$ is contained in $\cO_{B_n}\subset U_n$ for every $n=1,2,\ldots$ and $A\in \cO_{B_0}$ by construction.
\end{proof}

\subsection{Main result}
Recall that a subset of a topological space is called comeager if it  can be represented as a countable intersection of subsets with dense interior.
\begin{theorem}\label{thm:main}
Let $\cC$ be a Fra\"iss\'e class of finitely generated $L$-structures whose Fra\"iss\'e limit does not belong to $\cC$ and let  $(\bS, \cT)$ be the topological space of all countable not finitely generated $L$-structures on $\omega$ whose age is contained in $\cC$ and with the topology $\cT$ defined above.
Then, the set $\bF$ of all structures $M\in \bS$ which are universal and homogeneous for $\cC$ is comeager in $\bS$.
\end{theorem}
By Proposition~\ref{prop:baire} and Remark~\ref{rem:saturated}, this immediately implies  the  classical Fr\"aiss\'e theorem:
\begin{corollary}
Universal and homogeneous structures exist and are unique up to isomorphism.
\end{corollary}

\begin{proof}[Proof of Theorem~\ref{thm:main}]
For a structure  $A\in \cC$ denote by $\bF_A$ the set of all structures in $\bS$ which contain an isomorphic copy of $A$ as a substructure. For a quadruple $(A_1,A_2,a,b)$, where $A_1\subset A_2$ are structures in $\cC$, $a$ is a finite tuple of elements generating $A_1$, and $b$ is a tuple of elements of $\omega$ of the same length as $a$, denote by $\bF_{A_1,A_2,a,b}$ the set of all structures $S\in\bS$ such that one of the following holds:
\begin{itemize}
\item[(i)]  either the map $a\mapsto b$ cannot be extended to an embedding of $A_1$ into $S$ or
\item[(ii)] there is an embedding $\varphi:A_2\to S$ extending $a \mapsto b$.
\end{itemize}
By Remark~\ref{rem:saturated}, we can write
$$
\bF = \left(\bigcap_{A} \bF_A \right) \cap \left(\bigcap_{(A_1,A_2,a,b)} \bF_{A_1,A_2,a,b} \right).
$$
It suffices to show that the sets $\bF_A$ and $\bF_{A_1,A_2,a,b}$ are open and dense.

\vspace*{2mm}
\noindent
\textit{Claim 1:} The set $\bF_A$ is open. Consider a structure $S\in \bF_A$. That is, $S$ contains a substructure $A'$ isomorphic to $A$. Then, $S\in \cO_{A'} \subset \bF_A$ by definition, thus proving the claim.

\vspace*{2mm}
\noindent
\textit{Claim 2:} The set $\bF_A$ is dense. Consider a non-empty open set of the form $\cO_B$. Note that $|\omega\backslash \dom B| = \infty$ by Lemma~\ref{lem:O}.  By the joint embedding property, we can extend $B$ to a larger structure $D$ (with $\dom D\subset \omega$, $|\omega \backslash \dom D| =\infty$ and $D\in \cC$) which contains also an isomorphic copy of $A$. Furthermore, we can extend $D$ to a structure $S\in \bS$ with universe $\omega$; see Lemma~\ref{lem:O}. By construction, $S\in \cO_D\subset \cO_B \cap \bF_A$, thus proving the claim.

\vspace*{2mm}
\noindent
\textit{Claim 3:} The set $\bF_{A_1,A_2,a,b}$ is open. Let $S\in \bF_{A_1,A_2,a,b}$ be some structure. We consider two cases.

\vspace*{2mm}
\noindent
Case (i): The map $a\mapsto b$ cannot be extended to an embedding of $A_1$ into $S$. Let $B$ be the structure generated by the tuple $b$ in $S$. Clearly, $B\in \cC$. Then, for every $S'\in \cO_B$, the map $a\mapsto b$ cannot be extended to an embedding of $A_1$ into $S'$.  Hence, $S\in \cO_B \subset \bF_{A_1,A_2,a,b}$.

\vspace*{2mm}
\noindent
Case (ii): $\varphi$ is an embedding of  $A_2$ into $S$ that maps $a$ to $b$. Then, $\varphi(A_2)\in \cC$ is a substructure of $S$. The open set $\cO_{\varphi(A_2)}$ contains $S$ and is contained in $\bF_{A_1,A_2,a,b}$, thus proving the claim.

\vspace*{2mm}
\noindent
\textit{Claim 4:} The set $\bF_{A_1,A_2,a,b}$ is dense. Consider a non-empty open set of the form $\cO_D$. Note that $|\omega \backslash \dom D| = \infty$ by Lemma~\ref{lem:O}.   Our aim is to construct a structure $S\in \bF_{A_1,A_2,a,b} \cap \cO_D$.
Write $a= (a_1,\ldots,a_n)$ and $b= (b_1,\ldots,b_n)$. Further, let $b_1,\ldots,b_m$ be those elements from the tuple $b$ that are contained in $D$. Denote the substructure of $D$ generated by $b_1,\ldots,b_m$ by $D'$. Similarly, let  $A_1'$ be the substructure of $A_1$ generated by $a_1,\ldots,a_m$. Consider the following two cases.

\vspace*{2mm}
\noindent
Case (a): The map $a_1\mapsto b_1, \ldots, a_m \mapsto b_m$ cannot be extended to an isomorphism between $A_1'$ and $D'$. Then, for an arbitrary structure $S\in \cO_D$, the map $a\mapsto b$ cannot be extended to an embedding of $A_1$ into $S$, thus proving that $S\in \bF_{A_1,A_2,a,b} \cap \cO_D$.

\vspace*{2mm}
\noindent
Case (b):
There is an isomorphism between $A_1'$ and $D'$ extending the map $a_1\mapsto b_1, \ldots, a_m \mapsto b_m$. Amalgamate $A_2$ and $D$ along their common part $A_1'$ (identified with $D'$) to get a structure $F$. We denote by $\varphi$ the embedding of $A_2$ into $F$. By definition, $\varphi(a_1)=b_1,\ldots,\varphi(a_m)=b_m$. Using the fact that $|\omega \backslash \dom D| = \infty$,  we can choose $\dom F\subset \omega$  such that $F$ extends $D$, $|\omega \backslash \dom F| = \infty$ and $\varphi(a_{m+1}) = b_{m+1}, \ldots, \varphi(a_{n}) = b_{n}$. Finally, extend $F$ to a structure $S\in \bS$; see Lemma~\ref{lem:O}. The map $a\mapsto b$ extends to the embedding $\varphi: A_2 \to S$, thus showing that $S\in \bF_{A_1,A_2,a,b} \cap \cO_D$.
\end{proof}

\subsection{Remark on the topology} \label{subsec:rem_topology}
Assuming that all structures in $\cC$ are finite and the language $L$ is finite, it is possible to give an alternative description of the topology $\cT$ introduced above. Consider on $\bS$ the topology $\cT_0$ generated by the basic open sets $\cU_\varphi$ of the following form. Let $L(\omega)$ denote a language obtained from $L$ by adjoining a constant symbol for every element in $\omega$. For a quantifier-free $L(\omega)$-sentence $\phi$ we put
\[
\cU_\phi=\{M\in\bS\colon M\models\phi\}.
\]
It is easy to check that the sets of the form $\cU_\varphi$ (after removing those sets which are empty) form a base of some topology denoted by $\cT_0$. Indeed, if $A\in \cU_{\varphi_1} \cap \cU_{\varphi_2}$, then $A\in \cU_{\varphi_1\wedge \varphi_2} \subset \cU_{\varphi_1} \cap \cU_{\varphi_2}$. Note that the sets $\cU_\phi$ are clopen because the complement of $\cU_{\varphi}$ is $\cU_{\neg \varphi}$. In fact, $\cT_0$ is just the topology of pointwise convergence of structures. Namely, a sequence $A_1,A_2,\ldots\in\bS$ of structures on $\omega$ converges in this topology to an structure $A_\infty\in\bS$ iff for every arity $m\in\N$,  every $m$-ary relation symbol $r$, every $m$-variable function symbol $f$, and every constant symbol $c$ from the language $L$ with their respective realizations $r_n: \omega^m \to \{0,1\}$, $f_n:\omega^m\to\omega$, $c_n\in\omega$ in $A_n$, where $n\in\N\cup\{\infty\}$, we have
$$
\lim_{n\to\infty} r_n(x)= r_\infty(x),
\;\;
\lim_{n\to\infty} f_n(x)= f_\infty(x),
\;\;
\lim_{n\to\infty} c_n = c_\infty,
$$
for all tuples $x\in\omega^m$.

\begin{proposition}\label{prop:top_equal}
If the language $L$ is finite and all structures in $\cC$ are finite, then the topologies $\cT$ and $\cT_0$ coincide.
\end{proposition}
\begin{proof}
If $B\in \cC$ is some (finite) structure with $\dom B\subset \omega$, then (due to finiteness of $L$) there is a quantifier-free formula $\varphi$ in the language $L(\omega)$ expressing that a given structure $A$ on the universe $\omega$ coincides with $B$ on $\dom B$. Hence, we have found a set $\cU_{\varphi}$ which is equal to $\cO_B$. Hence, $\cT$ is weaker than $\cT_0$. To prove the converse, let $\varphi$ be any quantifier-free formula in the language $L(\omega)$. Consider some structure $A\in \cU_\varphi$.  Let $x_1,\ldots,x_n\in\omega$ be the constants appearing in $\varphi$ and consider a substructure $B$ generated by $x_1,\ldots,x_n$ in $A$. Clearly, $A\in \cO_B \subset \cU_\varphi$. Hence, $\cT_0$ is weaker than $\cT$.
\end{proof}

\section{Examples} \label{sec:examples}

\subsection{The random graph}
Let $\cC$ be the class of all finite graphs. It is a Fra\"iss\'e class whose Fra\"iss\'e limit is the random graph~\cite{erdoes_renyi, rado}.
The space $\bS$ consists of all graphs with vertex set $\omega$ and can be identified with a closed subset of $\{0,1\}^{\omega\times \omega}$. The latter space is endowed with the product topology (equivalently, the  topology of pointwise convergence), and the topology $\cT$ on $\bS$ that we defined above is just the induced topology. Theorem~\ref{thm:main} implies that in the space $\bS$ of all graphs on the vertex set $\omega$, the graphs isomorphic to the random graph form a comeager subset. This result, together with its analogue for the Urysohn space, is due to Vershik~\cite{vershik_random_and_universal,vershik04,vershik06}. Similarly, Theorem~\ref{thm:main} applies to $K_n$-free graphs, bipartite graphs, hypergraphs etc.

\subsection{Hall's group}
Consider the language $L=\{e, \circ, \phantom{}^{-1}\}$ and let $\cC$ be the class of all finite groups.  It is a Fra\"iss\'e class whose Fra\"iss\'e limit is the Hall universal  group $H$ introduced in~\cite{hall}. The space $\bS$ consists of all locally finite groups on the underlying set $\omega$.   By Proposition~\ref{prop:top_equal}, the topology $\cT$ on $\bS$ can be described as follows: a sequence of group structures converges, if $e, a\circ b$, $a^{-1}$ become eventually constant in this sequence, for all $a,b\in \omega$. Theorem~\ref{thm:main} yields the following
\begin{corollary}
In the space $\bS$ of all locally finite groups on a fixed countable set, those groups isomorphic to the Hall group $H$ form a comeager subset.
\end{corollary}

One may ask whether the isomorphism type of the Hall group is comeager even in the space $\bL$ of all (not necessarily locally finite) groups on $\omega$. It is natural to endow $\bL$  with the topology of pointwise convergence of group structures defined in the same way as in Section~\ref{subsec:rem_topology}. Let us show that $\bS$ is not dense in $\bL$, hence the answer is ``no''.  Let $M$ be a  finitely presented, simple, torsion-free (and hence infinite) group; see~\cite{burger_mozes}. Let $x_1,\ldots,x_k\neq e$ be the generators (assumed to be pairwise non-equal), $r_1(x_1,\ldots,x_k) = \ldots = r_s(x_1,\ldots,x_k)=e$ the relations defining $M$ and let $\phi$ be the sentence
\[
\left(\bigwedge_{i=1,\ldots, k} x_i\neq e \right) \wedge 
\left(\bigwedge_{j=1,\ldots, s}r_j(x_1,\ldots, x_k)=e\right)
\]
in the language $L(\omega)$. Then $\cU_\phi$ contains $M$ and hence is non-empty. But since $M$ is simple, any group satisfying $\phi$ contains a subgroup isomorphic to $M$. Thus, $\cU_\phi$ does not contain any locally finite group.

\subsection{Algebraic closure of \texorpdfstring{$\bF_p$}{Fp}}
Let $L=\{\underline{0}, \underline{1}, +, -, \cdot, \phantom{}^{-1}\}$ and let $\cC$ be the class of finite fields of fixed characteristic $p$. The Fra\"iss\'e limit of $\cC$ is the algebraic closure of $\bF_p$, denoted by $\overline{\bF}_p$.  Since $\overline{\bF}_p$ is locally finite and any quantifier-free sentence in the language of fields consistent with the theory of fields of fixed prime characteristic $p$ can be realized in $\overline{\bF}_p$ (see e.g. \cite[3.3.12, 3.3.13]{TZ}), this shows that the set $\bS$ of locally finite fields of characteristic $p$ is dense in the space $\bL$ of all countable fields of characteristic $p$ endowed with the topology of pointwise convergence of structures defined in the same way as in Section~\ref{subsec:rem_topology}.  Also, $\bS$ can be represented as an intersection of countably many open subsets of $\bL$ (which are also dense by the above). Indeed, for every $x\in\omega$ the set of all fields in $\bL$ for which $x$ is algebraic is open.
Using Theorem~\ref{thm:main} we obtain the following result:
\begin{corollary}
In the space $\bL$ of all countable fields of characteristic $p$ on a fixed countable universe, the set of fields isomorphic to $\overline{\bF}_p$ is comeager.
\end{corollary}

Note that the results in \citep{Ackerman_etal} and \citep{Droste_Kuske} do not apply to Hall's universal group or the algebraic closure of $\bF_p$.

\begin{remark}
In the  examples described above  the set of structures which are \emph{not} universal-homogeneous is also dense in $\bS$. This is not hard to see for the random graph and the algebraic closure of \texorpdfstring{$\bF_p$}{Fp}.  For the space of locally finite groups, we note that  any finite group $G$ can be embedded into the locally finite group $G \oplus (\bQ/\mathbb Z)$. From a theorem of Mazurkiewicz~\cite[Exercise 6.2.A(a) on p.~370]{engelking_book} it follows that the space of universal-homogeneous structures (endowed with the induced topology) is homeomorphic to the space of irrational numbers.
\end{remark}

\subsection{The rational Urysohn space}
A metric space in which all distances are rational is called a $\bQ$-metric space. It is natural to consider $\bQ$-metric spaces as structures over the language $L_\bQ=\{d_r\colon r\in \bQ, r\geq 0\}$, where $d_r$ is the binary relation expressing that the distance between two points equals $r$.   Theorem~\ref{thm:main} can be applied to the class  $\cC$  of all finite $\bQ$-metric spaces (up to isometry). The Fra\"iss\'e limit of $\cC$ is called the rational Urysohn space. The elements of the space $\bS$ can be identified with $\bQ$-metric spaces on $\omega$. A sequence $A_1,A_2,\ldots$ of such metric spaces converges to $A$ iff for every two points $x,y\in\omega$ there is a number $n_0(x,y)\in\omega$ such that the distance between $x,y$ in $A_n$ is the same as in $A$, for all $n\geq n_0(x,y)$. From Theorem~\ref{thm:main} we conclude that the isometry type of the rational Urysohn space is comeager in the space of all countable $\bQ$-metric spaces. The same considerations apply to metric spaces in which the distances belong to some given finite or countable set.

\subsection{The Urysohn space}
Vershik~\cite{vershik_random_and_universal,vershik04,vershik06} proved that in the space of all metrics on a fixed countable set, those metrics whose completion is isometric to the Urysohn metric space form a comeager subset. It seems that Theorem~\ref{thm:main} cannot be applied to recover this result directly because the class of all finite metric spaces is not countable. In fact, in Vershik's setting the elements of the comeager subset are not isometric to each other (just their closures are), which is different from the conclusion of Theorem~\ref{thm:main}.

\subsection{Further examples}
Theorem~\ref{thm:main} can be applied to many other Fra\"iss\'e classes, for example to finite linear orders (whose Fra\"iss\'e limit is the dense linear order $(\bQ, <)$),  finite Boolean algebras (whose Fra\"iss\'e limit is isomorphic to the Boolean algebra of clopen subsets of the Cantor space), finite posets, finite join semilattices~\cite{droste_semilattice}, finite distributive lattices,  and so on.

\section*{Acknowledgement}
We are grateful to Anton Bernshteyn, Tam\'as K\'atay, Alex Kruckman and Minh Chieu Tran for finding errors in the previous versions of the paper and suggesting how to fix them.
We thank Alex Kruckman, Lionel Nguyen Van Th\'e and Manfred Droste for pointing out to us most references listed in the introduction.






\bibliography{hall_group_bib}

\bibliographystyle{plainnat}

\end{document}